\documentclass[11pt,a4paper]{article}
\usepackage{graphicx}
\usepackage{geometry}
\usepackage[breaklinks=true,colorlinks=true,linkcolor=blue,urlcolor=blue,citecolor=blue]{hyperref}
  \expandafter\let\csname equation*\endcsname\relax
  \expandafter\let\csname endequation*\endcsname\relax
\usepackage{amsmath}
\usepackage{amssymb}
\usepackage{amsthm}
\usepackage{array}
\usepackage{tikz}
\usepackage{graphicx} 
\usepackage{cite}
\usepackage{bm}
\usepackage{textcomp}
\usepackage{subfig}

\usetikzlibrary{arrows}
\usetikzlibrary{positioning}
\tikzset{>=stealth}

\newtheorem{mydef}{Definition}
\newtheorem{mythm}{Theorem}
\newtheorem{myprop}{Proposition}
\newtheorem{lem}[mythm]{Lemma}
\newtheorem{rem}[mythm]{Remark}

\newtheorem{cor}[mythm]{Corollary}

\DeclareMathOperator{\rnk}{rank}
\DeclareMathOperator{\id}{id}

\DeclareMathOperator{\Img}{Im}

\newcommand{\lb}{\big[\!\!\big[}
\newcommand{\rb}{\big]\!\!\big]}

\title{Characterization of Exact Lumpability for Vector Fields on Smooth Manifolds}

\author{Leonhard Horstmeyer\thanks{Max Planck Institute for Mathematics in the Sciences, Inselstra\ss e 22, 04103 Leipzig, Germany.  \texttt{(horstmey@mis.mpg.de)}}\;\,\thanks{Basic Research Community for Physics, Mariannenstra\ss e 89, 04315 Leipzig, Germany.} 
\and 
Fatihcan M. Atay\thanks{Department of Mathematics, Bilkent University, 06800 Bilkent, Ankara, Turkey.  \texttt{(atay@member.ams.org)}}}
\date{}

\begin{document}
\maketitle
\begin{tikzpicture}[remember picture,overlay]
  \node [xshift=2.75cm,yshift=-5.5cm] at (current page.north west)
    [above right]
        {\scriptsize\begin{tabular}{l}Preprint. Final version in: \\Differ.Geom.Appl. \textbf{48} (2016) 46-60 \\
        doi: \ 10.1016/j.difgeo.2016.06.001\end{tabular}};
\end{tikzpicture}




\begin{abstract}
We characterize the exact lumpability of smooth vector fields on smooth manifolds.
We derive necessary and sufficient conditions for lumpability
and express them from four different perspectives,
thus simplifying and generalizing various results from the literature that exist for Euclidean spaces. 
We introduce a partial connection on the pullback bundle that is related to the Bott connection and behaves like a Lie derivative. The lumping conditions are formulated in terms of the differential of the lumping map, its covariant derivative with respect to the connection and their respective kernels. Some examples are discussed to illustrate the theory.
\end{abstract}

PACS numbers: 02.40.-k, 02.30.Hq, 02.40.Hw

AMS classification scheme numbers: 37C10, 34C40, 58A30, 53B05, 34A05

\vspace{2pc}
\noindent{\it Keywords}: lumping, aggregation, dimensional reduction, Bott connection

\section{Introduction}

Dimensional reduction is an important aspect in the study of smooth dynamical systems and in particular in modeling with ordinary differential equations (ODEs). Often a reduction can elucidate key mechanisms, find decoupled subsystems, reveal conserved quantities, make the problem computationally tractable, or rid it from redundancies. A dimensional reduction by which micro state variables are aggregated into macro state variables also goes by the name of \textit{lumping}. Starting from a micro state dynamics, this aggregation induces a \textit{lumped dynamics} on the macro state space. Whenever a non-trivial lumping, one that is neither the identity nor maps to a single point, confers the defining property to the induced dynamics, one calls the dynamics \textit{exactly lumpable} and the map an \textit{exact lumping}. 

Our aim in this paper is to provide necessary and sufficient conditions for exact lumpability of smooth dynamics generated by a system of ODEs on smooth manifolds. 
To be more precise, let $X$ and $Y$ be two smooth manifolds of dimension $n$ and $m$, respectively, with $0<m<n$. 
Let $p_X:TX\to X$ and $p_Y:TY\to Y$ be their tangent bundles, whose fibers we take as spaces of derivations, and let $v$ be an element of the smooth sections $\Gamma^\infty(X,TX)$ of $TX$ over $X$, i.e. smooth maps from $X$ to $TX$ satisfying $p_X\circ v=\id  _X$.
The integral curves $\Phi_t$ of $v$ satisfy the equation
\begin{equation} \label{eq:dx}
	\frac{d}{dt}\Big|_{t=s} \Phi_t(x)=v(\Phi_s(x)) \;.
\end{equation}
On a local coordinate patch $U\subseteq X$ we can write (\ref{eq:dx}) as $\dot x^i=v^i(x)$ so that we recover an ODE on that patch. 
Consider a smooth surjective submersion $\pi:X\to Y$ and let $\Theta_t(x)=\pi\circ\Phi_t(x)$. Since $\dim(Y) < \dim(X)$, the mapping $\pi$ is many-to-one, and hence is called a \emph{lumping}. The question is whether there exists a smooth dynamics on $Y$ that is generated by another system of ODEs,
\begin{equation*} 
	\frac{d}{dt}\Big|_{t=s} \Theta_t(x)=\tilde v(\Theta_s(x))
\end{equation*}
for some smooth vector field $\tilde v$ on $Y$.
If that is the case, we say that \eqref{eq:dx} is \emph{exactly lumpable} for the map $\pi$. Geometrically this means that $\tilde v$ and $v$ are $\pi$-related \cite{Lee03}.
The reduction of the state space dimension has been studied for Markov chains by Burke and Rosenblatt \cite{BurRos58,BurRos59} in the 1960s. Kemeny and Snell \cite{KemSne70} have studied its variants and called them weak and strong lumpability. 
Many conditions have been found, mostly in terms of linear algebra, for various forms of Markov lumpability \cite{KemSne70,BarTho77,RubSer89,RubSer91,RubSer93,BalYeo93,Buchho94e,JacGoe09,Jacobi10}. Since Markov chains are characterized by linear transition kernels, most of these conditions carry over directly to the case of linear difference and differential equations. 
In 1969 Kuo and Wei studied exact \cite{WeiKuo69a} and approximate lumpability \cite{WeiKuo69b} in the context of monomolecular reaction systems, which are systems of linear first order ODEs of the form $\dot x=Ax$. They gave two equivalent conditions for exact lumpability in terms of the commutativity of the lumping map with the flow or with the matrix $A$ respectively. 
Luckyanov \cite{Lucky83} and Iwasa \cite{Iwasa87} studied exact lumpability in the context of ecological modeling and derived further conditions in terms of the Jacobian of the induced vector field and the pseudoinverse of the lumping map. 
Iwasa also only considered submersions.
The program was then continued by Li and Rabitz et al., who wrote a series of papers successively generalizing the setting, but remaining in the Euclidean realm. 
They first constrained the analysis to linear lumping maps \cite{LiRabit89}, where they offered for the first time two construction methods in terms of matrix decompositions of the vector field Jacobian. These methods, together with the observability concept \cite{Luenb64} from control theory, were employed to arrive at a scheme for  approximate lumpings with linear maps \cite{LiRabit90}. 
They extended their analysis further to exact nonlinear lumpings of general nonlinear but differentiable dynamics \cite{LiRabTot94}, providing a set of necessary and sufficient conditions, extending and refining those obtained by Kuo, Wei, Luckyanov and Iwasa. By considering the spaces that are left invariant under the Jacobian of the vector field, they open up a new fruitful perspective, namely the tangent space distribution viewpoint.

The connection to control theory has been made explicit  in \cite{Coxson84}. Coxson notes that exact lumpability is an extreme case of non-observability, where the lumping map is viewed as the observable. She specifies another necessary and sufficient condition by stating that the rank of the observability matrix ought to be equal to the rank of the lumping map itself. The geometric theory of nonlinear control is outlined in, e.g., \cite{Isidor95}. There, Isidori considers sets of observables $h_i$ with values in $\mathbb R$ and their differentials $dh_i$. He discusses how to obtain the maximal observable subspace in an iterative fashion, where one consecutively constructs distributions that are invariant under the vector field and contain the kernel of the $dh_i$ \cite[p. 69]{Isidor95}. This distribution is constructed by means of $\mathcal L_v dh_i$, the Lie derivatives of $dh_i$. Although this theory is not concerned with the case of exact lumping, 
it follows that in the exactly lumpable case the maximal observable subspace is precisely the kernel of the $dh_i$ and the Lie derivatives $\mathcal L_v dh_i$ are just linear combinations of $dh_i$. (We obtain similar results, but allow for general maps, that are not necessarily $\mathbb R$-valued.) 

In this paper we tie together all these strands into one geometric theory of exact lumpability. The conditions obtained by Iwasa, Luckyanov, Coxson, Li, Rabitz, and T\'oth are contained in this framework. Instead of considering the distribution spanned by the differential of the lumping map, as is done in \cite{LiRabTot94} although not explicitly, we consider the vertical distribution which is defined by the kernel of the differential. We begin by stating the mathematical setting in Section \ref{ssc:Pre}. We then define the notion of exact smooth lumpability and provide two elementary propositions in terms of commutative diagrams  in Section \ref{ssc:Char}. In Section \ref{ssc:LVD} we characterize exact lumpability in terms of the vertical distribution and partial connections on it. In Section \ref{sc:PrEx} we investigate some properties of exact lumpings and illustrate them with examples.

\section{Characterization of Lumpability}
\label{sc:Char}
\subsection{Preliminaries}
\label{ssc:Pre}
As above, let $X$ and $Y$ be two smooth manifolds of dimension $n$ and $m$ and $p_X:TX\to X$ and $p_Y:TY\to Y$ be their tangent bundles, respectively.
The differential of a smooth manifold map $\pi:X\to Y$ 
at point $x$ is a $\mathbb R$-linear map $D \pi_x:T_xX\to T_{\pi(x)}Y$. 
For $w_x\in T_xX$ the vector $D\pi_x w_x$ can be defined via its action as a derivation $D \pi_x w_x [f]=w_x[f\circ \pi]$ on smooth test functions $f\in C^\infty(X,\mathbb R)$. We use square brackets to enclose the argument of the derivation. The map $\pi$ is a \textit{submersion} if $D \pi_x$ is surjective with constant rank for all $x\in X$.
We denote by $\pi^{-1}TY$ the pullback bundle whose fibers at $x$ are $T_{\pi(x)}Y$. There are two bundle maps associated to the differential. The first one is a manifold map $D\pi:TX\to TY$ which respects the vector bundle structure and satisfies $p_Y\circ D\pi=\pi\circ p_X$. The second one is a vector bundle homomorphism over the same base $D \pi\,:TX\to \pi^{-1}TY$. 
This latter one induces a $\mathcal C^\infty$-linear map on the vector fields $D \pi\, :\Gamma^{\infty}(X,TX)\to \Gamma^\infty(X,\pi^{-1}TY)$. All of these are denoted by $D\pi$ and the context will tell them apart. One can only define a vector field $\tilde w$ on $Y$ whenever there exists a unique vector $D \pi_x w(x)$ for all $x\in\pi^{-1}(y)$ and all $y\in Y$. 

A smooth regular distribution $S$ is a smooth subbundle locally spanned by smooth and linear independent vector fields \cite{Lee03,KoMiSl93}. The distribution $\ker D \pi\,=\bigsqcup_{x\in X}\ker D \pi_x$ can be shown to be smooth, where $\bigsqcup$ denotes disjoint union. This follows from the existence of a smooth local coframe (c.f. \cite{Lee03}) spanned by $m$ smooth 1-forms $(d\pi^1,\dots,d\pi^m)$ that annihilate $\ker D \pi\,$. The distribution $\ker D \pi\,$ is regular if and only if $\pi$ is a submersion. An integral submanifold $W$ of $S$ is an immersed submanifold of $X$ such that $TW\subseteq S|_W$. It has the maximal integral submanifold property if $TW= S|_W$ and $W$ is not contained in any other integral submanifold. Following Sussmann and Stefan \cite{Suss73,Stef74}, $S$ is integrable if every point of $X$ is contained in an integral submanifold with the maximal integral submanifold property. Frobenius theorem states that a regular distribution is integrable if and only if the space of its sections is closed under the Lie bracket, i.e., $S$ is involutive. 
The distribution $\ker D \pi\,$ is by construction an integrable distribution where $\{\pi^{-1}(x)\}_{x\in X}$ are the maximal integral submanifolds of maximal dimension.

Let $v$ and $w$ be two vector fields where $v$ generates the flow $\Phi$. The Lie derivative of $w$ in the direction $v$ is defined by
\begin{equation}
 \mathcal L_vw:=\frac{d}{dt}\Big|_{t=0} D\Phi_{-t}w\circ\Phi_{t}\;.\label{eq:lieder}
\end{equation}
The Lie derivative $\mathcal L_v:\Gamma^{\infty}(X,TX)\to \Gamma^{\infty}(X,TX)$ is a derivation on the $C^\infty$-module of vector fields. One can also show \cite{Lee03} that $\mathcal L_vw=\lb v,w\rb$, where $\lb\cdot,\cdot\rb:\Gamma^{\infty}(X,TX)\times \Gamma^{\infty}(X,TX)\to \Gamma^{\infty}(X,TX)$ is the Lie bracket.

A linear \emph{connection} on a vector bundle $E\to X$ is a map $\nabla^E:\Gamma^\infty(X,TX)\times \Gamma^\infty(X,E)\to \Gamma^\infty(X,E)$ which is tensorial in the first argument and for any $v\in \Gamma^\infty(X,TX)$ the map $\nabla^E_v:=\nabla^E (v,\cdot)$ is a derivation on $\Gamma^\infty(X,E)$. A \textit{partial connection} over a subbundle $S\subset TX$ is a map $\mathring \nabla^E:\Gamma^\infty(X,S)\times \Gamma^\infty(X,E)\to \Gamma^\infty(X,E)$. A notable partial connection is the Bott connection \cite{Bott72} defined over an integrable subbundle $S$ on the quotient bundle $Q=TX/S$. Let $\rho$ be the corresponding quotient map; then the connection is defined by
\begin{equation}\label{eq:Bott}
\mathring\nabla^{Q}_w\big[v\big]=\rho
\lb w,\rho^{-1}\big[v\big] \rb \;,
\end{equation}
where the right inverse $\rho^{-1}\big[v\big]=v'+w'$ picks out smoothly an arbitrary representative of the equivalence class, with $w'\in\ker \rho$. Since the Lie bracket is bilinear and $S$ is involutive, this is independent of the choice $w'$ and thus well defined. Only the term that is linear in $w$ survives the projection by $\rho$ and so the requirements of a veritable connection are satisfied. The partial connection can be completed to a full connection \cite{Tond97}. For example, one could introduce a Riemannian metric which splits $TX=S\oplus S^\perp$ and decomposes 
$g=g_S\otimes g_{S^\perp}$. The corresponding Levi-Civita connection $\bar\nabla$ restricted to $Q$ completes $\mathring\nabla^{Q}$ to a metric connection:
\begin{equation*}
 \nabla_w^{Q}=\mathring\nabla^{Q}_w + \bar\nabla^{|Q}_w\;.
\end{equation*}
This is sometimes called an \textit{adapted connection}.


\subsection{Lumpability and commutativity}
\label{ssc:Char}

In this section we state two  necessary and sufficient conditions for exact lumpability. Henceforth $\pi$ is a smooth surjective submersion and $v\in \Gamma^\infty(X,TX)$ is a smooth vector field generating the flow $\Phi:\mathcal T_X\subseteq \mathbb R \times X \to X$, where $\mathcal T_X:=\{(\mathcal T_x,x):\mathcal T_x\subseteq \mathbb R, x\in X\}$ is the domain of the flow and $\mathcal T_x$ contains an open interval around $0$. We denote by $\Phi_x:\mathcal T_x\to X$ the integral curves with starting point $x$, and by $\Phi_t: \mathcal X_t\to X$ the flow map parametrized by time, with $\mathcal X_t:=\{x\in X: t \in \mathcal T_x\}$ being the domain of definition. We start by giving a precise definition of lumpability.

\begin{mydef}[Exact Smooth Lumpability]
The system
\begin{equation}
	\frac{d}{dt}\Big|_{t=s} \Phi=v\circ\Phi_s \label{eq:dydef}
\end{equation}
is called exactly smoothly lumpable (henceforth exactly lumpable) for $\pi$ iff there exists a smooth vector field $\tilde v\in\Gamma^\infty(Y,TY)$ such that the dynamics of $\Theta=\pi\circ \Phi$ is governed by
\begin{equation}
\frac{d}{dt}\Big|_{t=s} \Theta = \tilde v\circ \Theta_s\;. \label{eq:lumpedode}
\end{equation}
\end{mydef}
The Picard-Lindel\"of theorem guarantees a unique solution of \eqref{eq:dydef} for sufficiently small times for all $x$, since $v$ is smooth and in particular Lipschitz. It exists for all times of definition $\mathcal T_{x}\subseteq \mathbb R$. Formally equation (\ref{eq:dydef}) should be understood as the pushforward of the section $\frac{\partial}{\partial t}$ on $\mathcal T_X$ by $\Phi$:
\begin{align*}
\frac{d}{dt}\Big|_{t=s}\Phi:&=\;
\big(D \Phi \big)\big|_{s}\,\frac{\partial}{\partial t}\;,
\end{align*}
and likewise for (\ref{eq:lumpedode}). The flow of the vector field $\tilde v\in\Gamma^\infty(Y,TY)$ is denoted by $\tilde \Phi:\mathcal T_Y\to Y$, where again $\mathcal T_Y:=\{(\widetilde{\mathcal T_y},y):(-\epsilon,\epsilon)\subseteq\widetilde{\mathcal T_y}\subseteq \mathbb R, y\in Y\}$ is the domain of the flow. There is no a priori connection between ${\mathcal T_x}$ and $\widetilde{\mathcal T_y}$. However, we will see later that Proposition \ref{pr:main2} relates the two.

\begin{myprop}
\label{pr:main1}
The system 
\eqref{eq:dydef} is exactly lumpable for $\pi$ iff there exists a smooth vector field $\tilde v\in\Gamma^\infty(Y,TY)$ such that
\begin{equation}
	D \pi_x v(x)=\tilde v\circ\pi(x) \label{eq:prop1}
\end{equation}
for all $x\in X$.
\end{myprop}

\begin{proof} 
Consider the time derivative of $\Theta$:
\begin{align*}
\frac{d}{d t}\Big|_{t=0} \Theta_x&=
D (\pi\circ\Phi_x)\big|_0\,\frac{\partial}{\partial t}=
D \pi_x v(x), \label{eq:compare1}
\end{align*}
By exact lumpability, $\Theta$ is generated by \eqref{eq:lumpedode}, so $\frac{d}{d t}\big|_{t=0} \Theta_x=\tilde v\circ \Theta_0(x)=\tilde v\circ \pi(x)$. Therefore, exact lumpability implies \eqref{eq:prop1}. On the other hand, if we demand \eqref{eq:prop1} for all $x$ and in particular for $\Phi_s(x)$, then
\begin{align*}
 D\pi_{\Phi_s(x)}v(\Phi_s(x))=\tilde v\circ\pi\circ \Phi_s(x)\;.
\end{align*}
The right hand side equals $\tilde v\circ \Theta_s(x)$ and the left hand side equals $\frac{d}{dt}\big|_{t=s}\Theta_t(x)$, which implies exact lumpability.
\end{proof}

\begin{rem}
Alternatively, we can say that \eqref{eq:dydef} is exactly lumpable for $\pi$ iff there exists a smooth vector field $\tilde v\in\Gamma^\infty(Y,TY)$ such that $\tilde v$ and $v$ are $\pi$-related. Proposition \ref{pr:main1} can be formulated as a commutative diagram
\begin{center}
\begin{tikzpicture}
    \node (Y1) at (0,0) {$Y$};
    \node[right=of Y1] (Y2) {$TY$};
    \node[below=of Y2] (X2) {$TX$};
    \node[below=of Y1] (X1) {$X$};
    \draw[->] (Y1)--(Y2) node [midway,above] {$\tilde v$};
    \draw[->] (X1)--(X2) node [midway,above] {$v$};
    \draw[->] (X2)--(Y2) node [midway,right] {$D\pi$}; 
    \draw[->] (X1)--(Y1) node [midway,left] {$\pi$};
\end{tikzpicture}
\end{center}
which reads $\tilde v (\pi(x))=D \pi_x v(x)$ for all $x\in X$.
\end{rem}


\begin{myprop}
\label{pr:main2}
The system 
\eqref{eq:dydef} is exactly lumpable for $\pi$ iff 
for all $y\in Y$ the time domain $\widetilde{\mathcal T}_y=\mathcal T_{x}$ is independent of the choice $x\in\pi^{-1}(y)$, and 
\begin{equation}
\tilde \Phi_t\circ\pi(x)=\pi\circ \Phi_t (x) \label{eq:prop2}
\end{equation} for all $x\in X$ and all times $t\in\widetilde{\mathcal T}_{\pi(x)}$.
\end{myprop}
\begin{proof}

One implication is obtained by taking time derivatives on both sides of \eqref{eq:prop2} at $t=0$ and using that $\tilde v$ is the generator of $\Tilde \Phi$. This gives rise to \eqref{eq:prop1} and by Proposition \ref{pr:main1} implies exact lumpability.
On the other hand, by the definition of exact lumpability, the curve $\Theta_x$ is an integral curve to $\tilde v$ for any $x$. There is another integral curve $\tilde\Phi_{\pi(x)}$ for $\tilde v$ which at $t=0$ coincides with $\Theta_x$. By the uniqueness of integral curves they must coincide, so $\tilde \Phi_{\pi(x)}(t)=\Theta_x(t)$ for all $t\in \mathcal T_x$ and all $x$. Since they are the same integral curves, $\widetilde{\mathcal T}_{\pi(x)}=\mathcal T_x$ for all $x$. This proves the proposition.
%
\end{proof}

\begin{rem}
Proposition \ref{pr:main2} can also be cast into a commutative diagram
\begin{center}
\begin{tikzpicture}
    \node (Y1) at (0,0) {$Y$};
    \node[right=of Y1] (Y2) {$Y$};
    \node[below=of Y2] (X2) {$X$};
    \node[below=of Y1] (X1) {$X$};
    \draw[->] (Y1)--(Y2) node [midway,above] {$\tilde \Phi_t$};
    \draw[->] (X1)--(X2) node [midway,above] {$\Phi_t$};
    \draw[->] (X2)--(Y2) node [midway,right] {$\pi$}; 
    \draw[->] (X1)--(Y1) node [midway,left] {$\pi$};
\end{tikzpicture}
\end{center}
which reads $\tilde \Phi_t\circ \pi=\pi\circ \Phi_t$ for all times of definition $t\in\widetilde{\mathcal T}_{\pi(x)}$ and all $x\in X$.
\end{rem}

\subsection{Lumpability and the vertical distribution}
\label{ssc:LVD}

In this section we discuss some relations between exact lumpability, invariant distributions, and the Bott connection. The lumping map $\pi:X\to Y$ gives rise to a subbundle $\ker D\pi\subseteq TX$ of the tangent bundle. This is called the \textit{vertical distribution}, which is integrable by construction and $\rho: TX\to TX/\ker D\pi$ is the corresponding quotient map. We start with a basic proposition.


\begin{myprop}
\label{pr:distr}
The distribution $\ker D\pi$ is invariant under the flow $\Phi$ iff the space of sections $\Gamma^\infty(X,\ker D\pi)$ is invariant under $\mathcal L_v$.
\end{myprop}
\begin{proof}
$\ker D\pi$ is invariant under the flow if $(D\Phi_t)_x(\ker D\pi)_x\subseteq (\ker D\pi)_{\Phi_t(x)}$ for all $x,t$ where it is defined. Since $\Phi_t$ is a diffeomorphism, this condition is equivalent to $(D\Phi_{-t})_{\Phi_t(x)}(\ker D\pi)_{\Phi_t(x)}\subseteq (\ker D\pi)_{x}$. So, for any $w\in\Gamma^\infty(X,\ker D\pi)$, we have $ (D\Phi_{-t}) w\circ \Phi_t \in \Gamma^\infty(X,\ker D\pi)$. Taking time derivatives and evaluating at 0, we obtain that the Lie derivative \eqref{eq:lieder} of $v$ in the direction of $w$ is again a section of $\ker D\pi$.  
\end{proof}

We would like to define a derivative of the differential $D\pi$ to find further conditions. 

\begin{mydef}[Covariant derivative of the differential] \label{df:LieDerOfDiff}
Let $\nabla^H$ be a connection on $H=(\pi\circ\Phi)^{-1}TY\otimes T^*(X\times \mathbb R)$ and $v\in\Gamma^\infty(X,TX)$ with flow $\Phi$. Then
\begin{equation}
\mathcal L^{\nabla}_v D \pi \,:= \nabla^H_{\frac{\partial}{\partial t}}\; D (\pi \circ \Phi)\;\big|_0 \label{eq:LieDerOfDiff}
\end{equation}
is the covariant derivative with respect to $\nabla^H$ of the differential $D\pi$ in the direction $\frac{\partial}{\partial t}=(0,\frac{\partial}{\partial t})\in T(X\times \mathbb R)$ .
\end{mydef}

The covariant derivative takes the place of $\frac{d}{dt}$ and ensures that the map 
$D(\pi\circ\Phi_t):TX\to (\pi\circ\Phi_t)^{-1}TY$ is differentiated properly and covariantly. It is worth noting that this object behaves like a Lie derivative as we will see in \eqref{eq:Comparison}, but since $D\pi$ is not a tensor one cannot define a proper Lie derivative. Nevertheless, we will use the similar notation. 

We shall make the connection to the Lie derivative more apparent. Let $V\to X$ be a vector bundle, $L:TX\to V$ a vector bundle homomorphism, and $\theta:X\to X$ a diffeomorphism. Then there exists an induced linear map $\theta_\sharp L:TX\to V$ of $L$: 
\begin{equation*}
\theta_\sharp L:=L\circ D\theta^{-1}\;.\label{eq:pushforwardlinmap2}
\end{equation*} 
Analogously to the Lie derivative \eqref{eq:lieder} of sections on the tangent bundle, we can then define \eqref{eq:LieDerOfDiff} as 
\begin{equation*}
\mathcal L^{\nabla}_v D \pi \,:= \nabla^H_{\frac{\partial}{\partial t}}\;\big(\Phi_{-t}\big)_\sharp D \pi \circ \Phi_t\;\big|_0\;\label{eq:LieDerOfDiff2}
 \end{equation*}
 with respect to $\nabla^H$.  

In Definition \ref{df:LieDerOfDiff} one needs to specify a covariant derivative. This is of course unfortunate, because there are many options.  However it turns out that we are fortunate nevertheless, because there is a good choice which turns out to be closely related to the Bott connection. 
Given a connection $\nabla^{E}$ on $E\to Y$ and a map $\pi:X\to Y$, there is a unique \cite{EelLem83} 
connection $\pi^*\nabla^{\pi^{-1}E}$ on $\pi^{-1}E \to X$, called the pullback connection
\begin{equation*}
\pi^*\nabla^{\pi^{-1}E}_v(s\circ \pi)=\big(\nabla^{E}_{D\pi v}s\big)\circ \pi ,
\end{equation*}
defined for sections $s\in \Gamma^\infty(Y,E)$ and extended locally to arbitrary sections $\sum_{a}c^a (s_a\circ\pi)\in\pi^{-1}E$ by linearity, where $c^a\in C^\infty(X,\mathbb R)$ for all $a$. 
Given a tensor product bundle $H=H_1\otimes H_2$, connections $\nabla^{H_1}$ and $\nabla^{H_2}$ on $H_1$ and $H_2$ respectively induce a connection on $H$ as follows:
\begin{equation}
 \nabla^H(s_1\otimes s_2)=\nabla^{H_1}s_1\otimes s_2 + s_1\otimes \nabla^{H_1}s_2,  \label{eq:tensorproductconnection}
\end{equation}
where $s_1$ and $s_2$ are sections on $H_1$ and $H_2$, respectively.
For the next proposition we require the connections to be torsion free. Recall that $\nabla$ is called torsion free if $\nabla_vw-\nabla_wv=\lb v,w\rb$.

\begin{lem} \label{pr:torsionfree}
 Let $g:M\to N$ and $\bar\nabla^{TN}$ be a torsion-free connection on $TN$. Then 
 \begin{equation}
 g^*\bar\nabla^{g^{-1}TN}_{ w } D g \;v - g^*\bar\nabla^{g^{-1}TN}_{ v } D g \;w = D g \lb w,v\rb, \label{eq:Antilinearity}
 \end{equation}
 where $v,w$ are sections on $TM$.
\end{lem}
\begin{proof}
 See page 6 of \cite{EelLem83}.
\end{proof}

\begin{myprop} 
\label{pr:simplify}
Let $\bar \nabla^{TY}$ and $\bar \nabla^{T^*(X\times\mathbb R)}$ be torsion-free connections and $\bar\nabla^H$ the tensor product connection \eqref{eq:tensorproductconnection}. Then
\begin{equation}
 \bar\nabla^{H}_{\frac{\partial}{\partial t}} D(\pi\circ\Phi)\;\big|_0\ w = \pi^*\bar\nabla^{\pi^{-1}TY}_w (D\pi v)\;. \label{eq:simplify}
\end{equation}
\end{myprop}
\begin{proof}
 The proof follows \cite{EelLem83} in the first part. 
With some abuse of notation, we use $w(x,t)=(w(x),0)\in T_{(x,t)}(X\otimes\mathbb R)$ and 
 $\frac{\partial}{\partial t}=(0,1)\in T_{(x,t)}(X\otimes\mathbb R)$. Then
\begin{align}
\bar\nabla^{H}_{\frac{\partial}{\partial t}} D(\pi\circ\Phi)\Big|_0 \! w
&=
\pi^*\bar\nabla^{(\pi\circ\Phi)^{-1}TY}_{\frac{\partial}{\partial t}} D(\pi\circ\Phi) w\Big|_0\!
-
D(\pi\circ\Phi)\;
\bar\nabla^{T(X\times\mathbb R)}_{\frac{\partial}{\partial t}} w \label{eq:temp1}
\end{align}
The second term vanishes because $\bar\nabla^{T(X\times\mathbb R)}=\bar\nabla^{TX\oplus T\mathbb R}$ and $w$ and $\frac{\partial}{\partial t}$ are orthogonal. 
Now we use Lemma \ref{pr:torsionfree} with $M=X\times \mathbb R$, $N=Y$, and $g=\pi\circ\Phi$, as well as the fact that $\bar\nabla^{TY}$ is torsion free, to obtain (p.6 \cite{EelLem83})
\begin{equation*}
\bar\nabla^{TY}_{D(\pi\circ\Phi) \frac{\partial}{\partial t}  } D(\pi\circ\Phi)w -
\bar\nabla^{TY}_{D(\pi\circ\Phi) w } D(\pi\circ\Phi)\frac{\partial}{\partial t} = D(\pi\circ \Phi) \Big[\!\!\Big[\frac{\partial}{\partial t},w\Big]\!\!\Big]. 
\end{equation*}
This vanishes because $w$ doesn't depend on $t$.
The pullback of this equation allows us to rewrite  \eqref{eq:temp1} as
\begin{align}
\bar\nabla^{H}_{\frac{\partial}{\partial t}} D(\pi\circ\Phi)\;\big|_0 w=&
\pi^*\bar\nabla^{(\pi\circ\Phi)^{-1}TY}_{ w } D(\pi\circ\Phi)\frac{\partial}{\partial t}\Big|_0\nonumber\\
=&
\pi^*\bar\nabla^{\pi^{-1}TY}_{ w } D\pi v . \nonumber
\end{align}
The last term is in principle over $T(X\times\mathbb R)$ but after having set $t=0$ we can omit the $T\mathbb R$ part.
\end{proof}

Lemma \ref{pr:torsionfree} and Proposition \ref{pr:simplify} show the analogy between $\mathcal L^{\bar\nabla}_v D \pi$ and the Lie derivative for torsion-free connections. Upon substitution of \eqref{eq:LieDerOfDiff} into \eqref{eq:simplify}, equation \eqref{eq:Antilinearity} reads
\begin{equation}
 \pi^*\bar\nabla^{\pi^{-1}TY}_{ v } D\pi w =  (\mathcal L^{\bar\nabla}_v D \pi) w + D\pi \mathcal L_vw,  \label{eq:Comparison}
\end{equation}
which should be compared to
\begin{equation*}
 \mathcal L_v \langle d\pi,w\rangle = \langle \mathcal L_v d\pi,w\rangle +  \langle d\pi, \mathcal L_v w\rangle, 
\end{equation*}
where $\pi:X\to \mathbb R$ is a real-valued function, $d\pi$ is a differential one-form, and $\langle\cdot,\cdot\rangle:T^*X\times TX\to \mathbb R$ is the natural pairing of tangent and co-tangent vectors. 

The linear map $\mathcal L^{\bar\nabla}_v D \pi :TX\to \pi^{-1}TY$ is a vector bundle homomorphism and the kernel $\ker \mathcal L^{\bar\nabla}_v D \pi$ is a smooth distribution, which can be checked by viewing $\mathcal L^{\bar\nabla}_v D \pi$ as a differential one-form: On each pullback patch $U\cap \pi^{-1}V\subseteq X$ with local coordinates $\tilde \psi:V\subseteq Y\to \mathbb R^m$,
one constructs locally a set of one-forms
\begin{equation}
\sigma^a:=
(\mathcal L^{\bar\nabla}_v D \pi)^a=
 d(D\pi v)^a+\bar\Gamma^a_{bc}(D\pi v)^cd\pi^b\label{eq:coframe}
\end{equation}
where $a,b,c$ are the indices of the local coordinates and $\bar\Gamma^a_{bc}$ is the Christoffel symbol of $\bar \nabla$. Here and in the remainder of the article, we use the convention that repeated indices are summed over, unless stated otherwise. Since $\pi$ has full rank, $(\sigma^1,\dots,\sigma^m)$ spans a smooth $m$-dimensional local co-frame. We have $\langle\sigma^a,w\rangle=((\mathcal L^{\bar\nabla}_v D\pi) w)^a$; so, this co-frame annihilates vectors in $\ker\mathcal L^{\bar\nabla}_v D\pi$.

The motivation for the Definition \ref{df:LieDerOfDiff} partly stems from the following two propositions:

\begin{myprop}\label{pr:LieDerOfDiff}
The distribution $\ker D\pi$ is invariant under the flow $\Phi_t$ iff the space of sections $\Gamma^\infty(X,\ker D\pi) \subseteq \Gamma^\infty(X,\ker\mathcal L^{\bar\nabla}_v D \pi)$.
\end{myprop}

\begin{proof}
By Proposition \ref{pr:distr}, the distribution $\ker D\pi$ is invariant under the flow $\Phi_t$ iff the space of sections $\Gamma^\infty(X,\ker D\pi)$ is invariant under $\mathcal L_v$. 
By \eqref{eq:Comparison}, if $w\in \Gamma^\infty(X,\ker D\pi)$ then $(\mathcal L^{\bar\nabla}_v D \pi)w=0 \iff \mathcal L_vw=0$.
\end{proof}

A slightly stronger version that implies Proposition \ref{pr:LieDerOfDiff} is the following.
\begin{myprop}\label{pr:InvarUnderLvDpi}
The distribution $\ker D\pi$ is invariant under the flow $\Phi_t$ iff $\ker D\pi \subseteq\ker\mathcal L^{\bar\nabla}_v D \pi$.
\end{myprop}

\begin{proof}
$\ker D\pi$ is invariant under the flow if $(D\Phi_t)_x(\ker D\pi)_x\subseteq (\ker D\pi)_{\Phi_t(x)}$ for all $x,t$ where it is defined. So, $(D\pi)_{\Phi_t(x)}(D\Phi_t)_x \,w_x=0$ for $w_x\in(\ker D\pi)_x$, or in other words $(D\pi)_{\Phi_t(x)}(D\Phi_t)_x $ maps $(\ker D\pi)_x$ to
$(\ker D\pi)_{\Phi_t(x)}$ so that $D(\pi\circ \Phi_t)w$ remains 0 for any $w\in \ker D\pi $. In infinitesimal terms this means that the covariant derivative \eqref{eq:LieDerOfDiff} vanishes, $\bar \nabla^H_{\frac{\partial}{\partial t}}\,D(\pi\circ\Phi_t) w \big|_0=(\mathcal L^{\bar\nabla}_v D \pi)w=0$ on $w$. 
\end{proof}

%
%

We would now like to define a partial connection on the pullback bundle $\pi^{-1}TY$ over sections of $\ker D\pi$. The next proposition establishes an isomorphism that will help us define the partial connection.

\begin{myprop}
There is a vector bundle isomorphism $\varphi:\pi^{-1}TY\to TX/\ker D\pi $.
\end{myprop}
\begin{proof}
We shall show that on each fiber $\varphi_x:T_{\pi(x)}Y\to T_xX/\ker D\pi_x$ is a vector space isomorphism. Let $\tilde v\in T_{\pi(x)}Y$. We fix local coordinates and denote the Jacobian of $\pi$ by $M^a_i=\frac{\partial \pi^a}{\partial x^i}$. There exists a unique pseudoinverse \cite{Penr55} $M^+$ such that $M^+M:T_xX\to (\ker M)^\perp$ is an orthogonal projection and $MM^+=\id_{T_{\pi(x)}X}$.
We show that $\varphi_x:\tilde v\mapsto \big[M^+\tilde v\big]$ is one-to-one and onto. Suppose $\varphi_x \tilde v=\varphi_x \tilde v'$, then $M^+\tilde v-M^+\tilde v'=w$ and $w\in \ker M$. 
Applying $M$ yields $\tilde v=\tilde v'$. 
To show surjectivity, we construct $\tilde v=M\big[v\big]$, which is the element that maps to $\big[v\big]$. So $\varphi_x$ is clearly a fiberwise isomorphism and $\varphi$ is a vector bundle isomorphism. In fact,
\begin{equation}
 \varphi^{-1}\circ \rho=D\pi \label{eq:dfandf}
\end{equation} 
is the differential.
\end{proof}

\begin{mydef}[Lumping Connection]  \label{partial_connection}
 The Lumping Connection is a partial connection 
 $$\mathring\nabla^{\pi^{-1}TY}:\Gamma^{\infty}(X,\ker D\pi)\times \Gamma^{\infty}(X,\pi^{-1}TY)\break\to \Gamma^{\infty}(X,\pi^{-1}TY)$$ 
defined by
\begin{equation}
 \mathring\nabla^{\pi^{-1}TY}_w\tilde v:=D\pi\lb w,v\rb,\label{eq:conn}
\end{equation}
where $w\in \Gamma^{\infty}(X,\ker D\pi)$, $v\in \Gamma^{\infty}(X,TX)$ and $\tilde v=D\pi  v\in  \Gamma^{\infty}(X,\pi^{-1}TY)$.
\end{mydef}

Definition \ref{partial_connection} indeed satisfies the requirements of a connection: Let $f\in \mathcal C^\infty(Y,\mathbb R)$ be a test function on $Y$. 
Recall that $D\pi w[f]:=w[f\circ \pi]$; so,
\begin{equation}
 D\pi\lb w,v\rb [f]=w[v[f\circ\pi]]-v[w[f\circ\pi]]\label{eq:DpiOnLie}\;.
\end{equation}
If $w \in\Gamma^{\infty}(X,\ker D\pi)$ then the second term vanishes. The first term is linear in $w$ and a derivation in $D\pi v$.

\begin{myprop}
\label{pr:BottvsLeo}
The connection defined in \eqref{eq:conn} is related to the Bott connection \eqref{eq:Bott} through the commutative diagram
\begin{center}
\begin{tikzpicture}
    \node (Y1) at (0,0) {$\pi^{-1}TY$};
    \node[right=of Y1] (Y2) {$TX/\ker D\pi$};
    \node[below=of Y2] (X2) {$TX/\ker D\pi$};
    \node[below=of Y1] (X1) {$\pi^{-1}TY$};
    \draw[->] (Y1)--(Y2) node [midway,above] {$\varphi$};
    \draw[->] (X1)--(X2) node [midway,above] {$\varphi$};
    \draw[->] (X2)--(Y2) node [midway,right] {$\mathring\nabla^{TX/\ker D\pi}$}; 
    \draw[->] (X1)--(Y1) node [midway,left] {$\mathring\nabla^{\pi^{-1}TY}$};
\end{tikzpicture}
\end{center}
where $TX/\ker D\pi=Q$ in \eqref{eq:Bott}.
\end{myprop}

\begin{proof}
By \eqref{eq:dfandf},
\begin{equation*}
\varphi\mathring\nabla^{\pi^{-1}TY}_w\tilde v=\varphi\circ \varphi^{-1}\circ \rho \lb w,\rho^{-1}(\varphi(\tilde v))\rb=\mathring\nabla^{TX/\ker D\pi}_w\varphi(\tilde v).
\end{equation*}
Therefore, $\varphi\mathring\nabla^{\pi^{-1}TY}_w\tilde v=\mathring\nabla^{TX/\ker D\pi}_w\varphi(\tilde v)$ for any $w\in \Gamma^{\infty}(X,\ker D\pi)$.
\end{proof}

\begin{myprop}\label{pr:completionofpartialconnection}
 Let $\bar\nabla^{TY}$ be a torsion-free connection on $TY$. Then
$
 \pi^{*}\bar\nabla^{\pi^{-1}TY}\;\nonumber
$
 completes the partial connection \eqref{eq:conn}.
\end{myprop}

\begin{proof}
Let $w \in \Gamma^{\infty}(X,\ker D\pi)$. By \eqref{eq:Antilinearity} we have 
\begin{align}
\pi^{*}\bar\nabla^{\pi^{-1}TY}_{w}D\pi v=D\pi\lb w, v\rb=\mathring\nabla^{\pi^{-1}TY}_{w}D\pi v, \nonumber
\end{align}
and therefore
$ \pi^{*}\bar\nabla^{\pi^{-1}TY} 
=  \mathring\nabla^{\pi^{-1}TY} +\pi^{*}\bar\nabla^{\pi^{-1}TY}\big|_{(\ker D\pi)^\perp}$.
 \end{proof}

We now connect all of these concepts to exact lumpability.

%

\begin{mythm}
\label{pr:main3}
The system 
\eqref{eq:dydef} is exactly lumpable for $\pi$ iff 
$\Gamma^\infty(X,\ker D\pi)$ is invariant under $\mathcal L_v$.
\end{mythm}
\begin{proof}
First we show that exact lumpability implies the invariance of \break $\Gamma^\infty(X,\ker D\pi)$ under $\mathcal L_v$. By exact lumpability, we know from (\ref{eq:prop1}) that there is a vector field $\tilde v$ such that $v[f\circ \pi]=\tilde v[f]\circ \pi$ for any test function $f\in \mathcal C^\infty(Y,\mathbb R)$. Substituting this condition into \eqref{eq:DpiOnLie} yields
\begin{align*}
D \pi\, \lb v,w\rb[f]&=v[w[f\circ\pi]]-w[\tilde v[f]\circ \pi] 
\end{align*}
The right hand side equals $v[D\pi w [f]]-D\pi w[\tilde v[ f]]$. So the left hand side vanishes for $w\in \Gamma^\infty(X,\ker D\pi)$.

Secondly we show that exact lumpability is implied by the invariance of \break $\Gamma^\infty(X,\ker D\pi)$ under $\mathcal L_v$.  We want to define the vector field $\tilde v$ as a smooth function of $y$ such that $\tilde v_{\pi(x)}=D \pi_x v(x) $ for all $x\in X$. 
This would imply exact lumpability due to (\ref{eq:prop1}). If $D \pi_x v(x)$ is constant along the fibers $x\in\pi^{-1}(y)$, then $\tilde v$ is well defined everywhere modulo smoothness, since $\pi$ is surjective.
We consider a vector field $w \in \Gamma^{\infty}(X,\ker D\pi)$ tangent to the fibers. By Proposition \ref{pr:completionofpartialconnection} the covariant derivative $\pi^* \bar\nabla^{\pi^{-1}TY}_w D\pi v=\mathring\nabla^{\pi^{-1}TY}_w D\pi v= D\pi\lb w,v\rb=0$ vanishes if $\Gamma^\infty(X,\ker D\pi)$ is invariant under $\mathcal L_v$. 
 
It remains to show that $\tilde v$ is a smooth function of $y$. This is the case if for any smooth curve $\tilde \gamma_y:(-\epsilon,\epsilon)\to Y$ the composition $\tilde v\circ\tilde \gamma_y$ is a smooth function in time.
But any such curve can be viewed as the composition of $\pi$ with a curve $\gamma_x:(-\epsilon,\epsilon)\to X$, where $\pi(x)=y$. 
Since for any $\gamma_x$ the equality $\tilde v\circ \pi \circ \gamma_x=D \pi\, v \circ \gamma_x$ holds, and since the right hand side is a composition of smooth functions and is thus also smooth, it follows that $\tilde v$ must be smooth. 
\end{proof}

\begin{cor}
\label{cr:InvFlow}
The system
\eqref{eq:dydef} is exactly lumpable for $\pi$ iff $\ker D\pi$ is invariant under the flow $\Phi$.
\end{cor}
\begin{proof}
This follows immediately from Proposition \ref{pr:distr}.
\end{proof}

\begin{cor}\label{cr:exlumpAndInvar}
The system
\eqref{eq:dydef} is exactly lumpable for $\pi$ iff $\ker D\pi \subseteq \ker\mathcal L_v^{\bar \nabla} D\pi$.
\end{cor}
\begin{proof}
 This follows from Proposition \ref{pr:InvarUnderLvDpi}
\end{proof}

We make the connection to control theory by introducing the 2-\textit{observability map}:
\begin{equation*}
\mathcal O_2:=\left(\begin{array}{c}
 D \pi\,  \\
 \mathcal L^{\bar\nabla}_vD \pi\, 
\end{array}\right):TX\to \pi^*TY\oplus \pi^*TY\;,
\end{equation*}
as the mapping  
\begin{equation*}
 v\mapsto (D\pi\oplus \mathcal L^{\bar\nabla}_vD \pi)(v\oplus v)
\end{equation*}
%
The $n$-observability map $\mathcal O_n:TX\to \bigoplus^n\pi^*TY$ is defined analogously with higher-order Lie derivatives. 
In the linear case, where $\dot x=v(x)=A x$ and $\pi(x)=C x$, we have $D\pi=C$, $\mathcal L^{\bar\nabla}_vD \pi=CA$, and $\mathcal O_2 = \left( \begin{array}{c}
  C\\
  CA
 \end{array}
\right)$; 
furthermore, $\mathcal O_n$ is just the standard observability matrix
familiar from linear control theory \cite{Trente12}, where the system is called \emph{observable} if $\rnk  \mathcal O_n=n$. 

\begin{myprop}
\label{pr:main4}
The system 
\eqref{eq:dydef} is exactly lumpable for $\pi$ iff  $\rnk  \mathcal O_2=\rnk D \pi\,$.
\end{myprop}

\begin{proof}
We consider the situation locally. Let  $\tilde\psi:V\subseteq Y\to \mathbb R^m$ be local coordinates on a patch $V\subseteq Y$, indexed by $a,b$ and $\psi:U\cap\pi^{-1}V\to\mathbb R^n$ coordinates on a pullback patch indexed by $i$. The rank of $\mathcal O_2$ is equal to the rank of $D\pi$ if and only if 
\begin{equation}
(\mathcal L^{\bar\nabla}_vD \pi)_i^a=\sum_b\phi^a{}_b (D\pi)_i^b\label{eq:matrixlumping}
\end{equation}
with smooth coefficient functions $\phi^a{}_b$. Now $w\in\ker D \pi\,$ implies $w\in\ker \mathcal L^{\bar\nabla}_vD \pi$, which implies exact lumpability  by Proposition \ref{cr:exlumpAndInvar}. 
On the other hand, considering the local coordinate form \eqref{eq:coframe} of $\mathcal L^{\bar\nabla}_vD \pi$ and demanding the system to be exactly lumpable, 
\begin{equation*}
(\mathcal L^{\bar\nabla}_vD \pi)_i^a =\frac{\partial}{\partial x^i}(\tilde v^a\circ \pi)+\bar\Gamma^a_{bc}(\tilde v^c\circ \pi)\frac{\partial \pi^b}{\partial x^i}=\sum_b\left(\frac{\partial \tilde v^a}{\partial y^b}+\bar\Gamma^a_{bc}\tilde v^c\right)\circ \pi\;(D\pi)_i^b,
\end{equation*}
which is of the form \eqref{eq:matrixlumping} and thus implies that $\rnk \mathcal O_2=\rnk D\pi$.
%
\end{proof}


\begin{cor}
 The system \eqref{eq:dydef} is exactly lumpable iff locally:
 \begin{equation*}
\bigwedge_{b=1}^m (D\pi)^b \wedge d\left( D\pi v\right)^a=0\qquad \forall \; a \in \{1,\dots, m\}.
 \end{equation*}
\end{cor}
\begin{proof}
Proposition \ref{pr:main4} states that the local condition \eqref{eq:matrixlumping} is necessary and sufficient for exact lumpability. So, the vectors $(D\pi)^a$ and $(\mathcal L_v^{\bar \nabla}D\pi)^b$ are linearly dependent. However, from \eqref{eq:coframe} it is seen that the second summand of $(\mathcal L_v^{\bar \nabla}D\pi)^b$ is already proportional to $(D\pi)^a$ , with the proportionality constant given by the Christoffel symbol. Hence, only the first summand $d\left( D\pi v\right)^a$ has to be checked for linear dependence.
\end{proof}

\section{Properties and Examples}
\label{sc:PrEx}

We next discuss some properties of exactly lumpable systems and illustrate them with examples. A very prominent class of submersions are fiber bundles $\pi:X\to Y$, and our examples are fiber bundle maps mostly over the 2-sphere $Y=S^2$. We begin by relating lumpability to the theory of integrable systems. Recall that a \emph{first integral} for the dynamics $v$ is a function $I:X\to\mathbb R$ such that $v[I]=0$.

\begin{myprop} \label{pr:frstintgrl}
Any system with a first integral $I$ of rank 1 is exactly lumpable.
\end{myprop}

\begin{proof}
Since $\rnk DI=1$, the quotient map $\pi=I$ is submersive. There exists a vector field $\tilde v=0$ on $\Img(I)$ such that $DI v=v[I]=0=\tilde v\circ I$. Thus, $v$ is exactly lumpable for $I$.
\end{proof}

\begin{rem}
Proposition \ref{pr:frstintgrl} also holds true if we relax the condition that exact lumpings have to be submersive and allow for target manifolds that have boundaries or are singular in other ways but can nevertheless be endowed with a smooth structure.
\end{rem}

In order to illustrate Proposition \ref{pr:frstintgrl}, we consider as an example the geodesic flow on the 2-sphere, which is generated by a vector field on the tangent bundle $TS^2$. We embed $TS^2\hookrightarrow\mathbb R^6$ by $(x,v)\mapsto (X,V)\in \mathbb R^3\times \mathbb R^3$, together with the requirement that the Euclidean dot products for $X$ and $V$ satisfy $X\cdot X=1$ and $X\cdot V=0$. Then,
\begin{equation}
\begin{array}{ccl}
   \frac{d}{dt}X_i&=& V_i\\
  \frac{d}{dt}V_i&=&-(V\cdot V)X_i
\end{array}\label{eq:geoflowS2}
\end{equation}
generates the geodesic flow \cite{MeHaOf09}. 
There is a stationary submanifold $\Omega=\{(X,V)\in  TS^2: V=0\}$.

We will use Proposition \ref{pr:frstintgrl} to show that the geodesic flow \eqref{eq:geoflowS2} on $TS^2\backslash\Omega$ is exactly lumpable for $I:TS^2\to \mathbb R$, given by $I(X,V)=V\cdot V$.
First we note that $I$ is a first integral to \eqref{eq:geoflowS2}, which can easily be seen by differentiating $I$ with respect to time and using $X\cdot V=0$. The geodesic flow can be viewed as a Hamiltonian flow whose energy is given by $\frac{1}{2}V\cdot V$. The rank of $I$ is 1, except on the stationary submanifold $\Omega$, where it equals 0. Hence, $I$ is submersive on $TS^2\backslash\Omega$  and satisfies $v[I]=0$. Therefore, the dynamics is exactly lumpable for $I$ by Proposition \ref{pr:frstintgrl}.

As a consequence of the energy conservation, the geodesic flow is just considered on one energy shell, say $V\cdot V=1$; so it effectively takes place on the unit tangent bundle $UTS^2\to S^2$.

\begin{myprop}\label{pr:orbitspace}
Any dynamics $v$ is exactly lumpable for the quotient map $\pi:X\to X/\Phi$ to the orbit space.
\end{myprop}

\begin{proof}
 The kernel of $\pi$ is simply the distribution spanned by $v$. This is trivially invariant under the flow $\Phi$ generated by $v$, since $D\Phi_sv=v\circ \Phi_s$ by definition and $v=D\Phi_t\frac{\partial}{\partial t}\Big|_0$.
Exact lumpability then follows from Corollary \ref{cr:InvFlow}. 
\end{proof}
To exemplify this Proposition, we now consider the geodesic flow on the unit tangent bundle of the 2-sphere $UTS^2$. We claim that it is exactly lumpable for the cross product $(X,V)\mapsto X\times V\in S^2$ and use the above Proposition to show this.

There is an isomorphism \cite{MeHaOf09} between the unit tangent bundle $UTS^2$ and $SO(3)$, given by $(X,V)\mapsto M$, where $M_{i1}=X_i$, $M_{i2}=V_i$, and $M_{i3}=(X\times V)_i$, or in compressed notation $M=(X|V|X\times V)$. So, for any $p\in S^2$ this matrix maps to another point $y=M\cdot p\in S^2$, and there is a collection of lumping candidates indexed by $p$. We choose $p=(0,0,1)$ and calculate the vector field induced by $\pi(X,V)=M(X,V) \;p=X\times V$:
 \begin{equation*}
  \sum_{i=1}^3\frac{\partial \pi}{\partial X_i}\frac{d}{dt}X_i+\sum_{i=1}^3\frac{\partial \pi}{\partial V_i}\frac{d}{dt}V_i=(0,0,0)\;.
 \end{equation*}
Thus, the dynamics \eqref{eq:geoflowS2} on $UTS^2$ lies in the kernel of $D\pi$. But $\pi$ is surjective onto $S^2$, it has constant rank, and $\dim\ker D\pi=1$. So, the vector field and hence the flow is parallel to the fibers and every point on $S^2$ corresponds to a flowline of the geodesic flow. This is illustrated in Figure \ref{fg:geoflow}. By Proposition \ref{pr:orbitspace}, $\pi$ is an exact lumping.

  \begin{figure}[!ht]
    \subfloat[Side View\label{sbfig:1}]{%
      \includegraphics[width=0.5\textwidth]{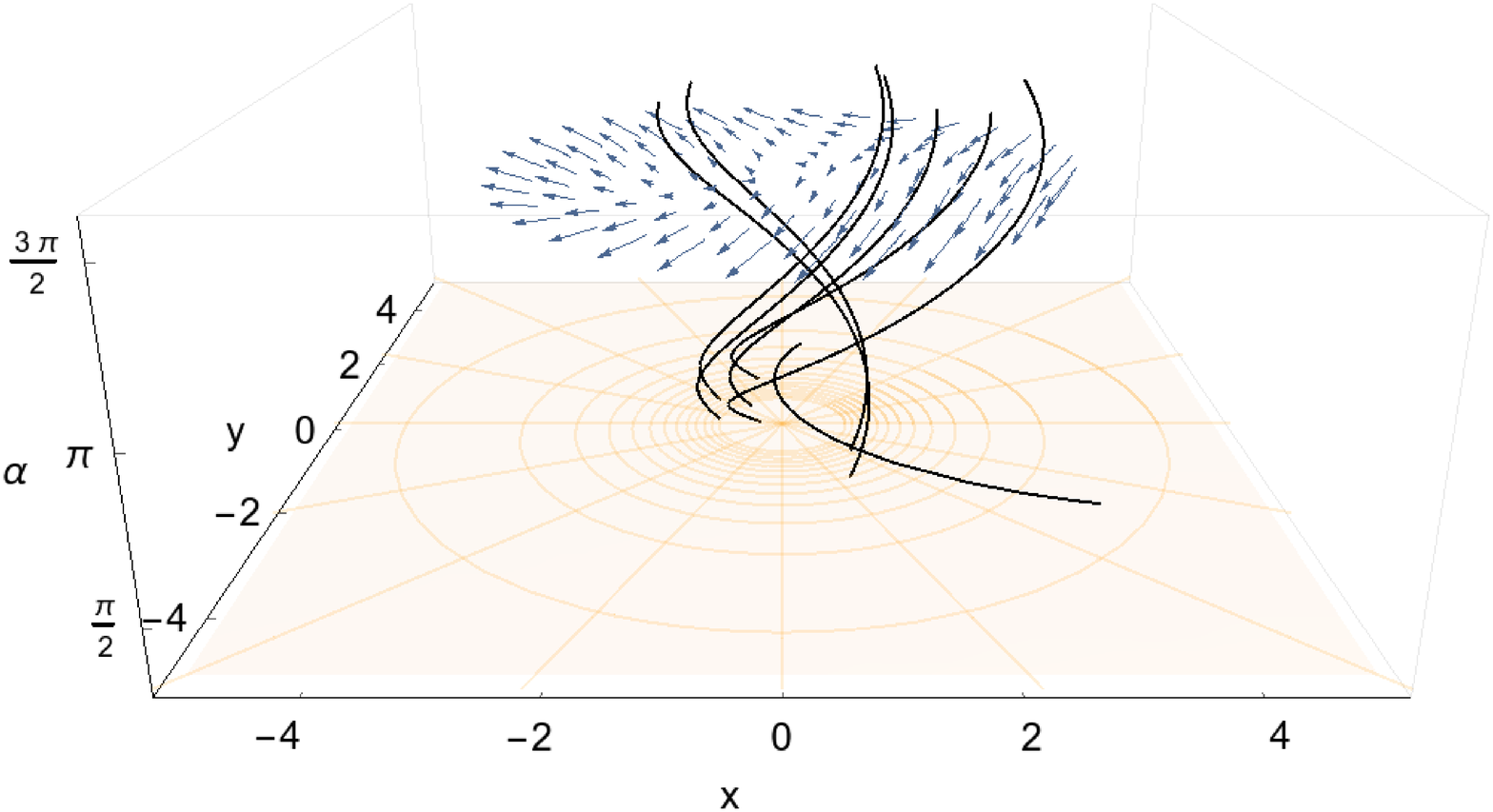}
    }
    \subfloat[Bird's Eye View\label{sbfig:2}]{%
      \includegraphics[width=0.5\textwidth]{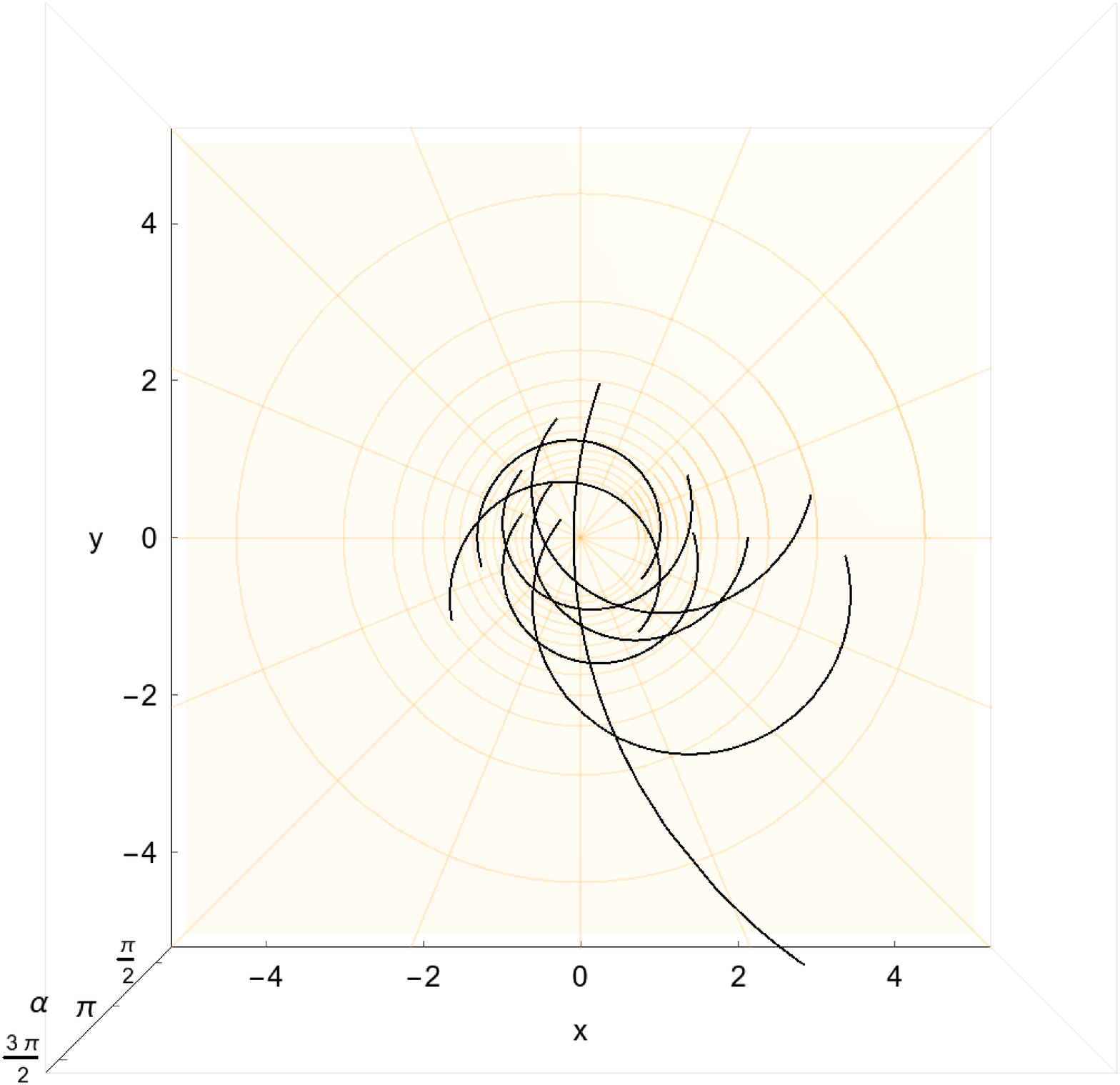}
    }
    \caption{We choose local coordinates $(x,y,\alpha)\in \psi(U)$, where $U$ is the unit tangent bundle restricted to the north pole $N\subset S^2$. The function $\psi$ acts by stereographic projection on the 2-sphere and maps the unit tangent vector $v$ to an angle $\alpha\in [0,2\pi)$, which is the angle enclosed by the $x$-direction and the push forward of $v$ under the stereographic projection. We depict fibers of the projection $\pi$ in the range $\pi/2\leq \alpha \leq 3 \pi/2$ from two different perspectives, indicating also the flow field in Figure \ref{sbfig:1}. The longitudes and latitudes of the sphere are seen on the bottom of the figures for reference.  }
    \label{fg:geoflow}
  \end{figure}

%

We next discuss the relation of lumpability to the symmetries of the system.
We shall show that the proper action of a Lie group that is compatible with the vector field results in an exact lumping; 
however, the converse is not true. Let $G$ be a finite Lie group with Lie algebra $\mathfrak{g}$. We denote by $A:G\to \text{Diff}(X)$ the left action of the Lie Group on $X$ and $a:\mathfrak g\to \Gamma^\infty(X,TX)$ the corresponding action of the Lie algebra. The action on the whole algebra is denoted by $\mathcal D=a(\mathfrak g)$. 
\begin{myprop}
\label{pr:main5}
If $\mathcal D$ is invariant under $\mathcal L_v$ and $G$ acts properly and freely, then $v$ is exactly lumpable for the quotient map $\pi:X\to X/G$. 
\end{myprop}

\begin{proof}
By the quotient manifold theorem \cite{Lee03} the quotient map of a proper and free Lie group action is a submersion and the quotient space has a natural smooth manifold structure. The vector fields that generate the action are annihilated by the differential of the quotient map; therefore, $\mathcal D=\Gamma^\infty(X,\ker D\pi)$ and so Proposition \ref{pr:main3} implies exact lumpability. 
\end{proof}
%
%
%

The converse statement to Proposition \ref{pr:main5} is not true. Given a vector field $v$ and a lumping $\pi$, the level sets need not be orbits of a proper and free Lie group action. The integrable distribution of a free Lie group action is spanned by its linearly independent generators making $\mathcal D$ a finitely generated submodule of the sections of $TX$. There are many integrable distributions that are not finitely generated and thus do not stem from a Lie Group action. Any section of such a distribution gives rise to a lumping that does not stem from a Lie group action.

The Hopf fibration over $S^2$,
\begin{equation*}
S^1\hookrightarrow S^3\stackrel{\pi}{\longrightarrow}S^2,
\end{equation*}
illustrates Proposition \ref{pr:main5}. We use the formulation of the Hopf map in terms of the quaternions $\mathbb H=(\mathbb R^4,\star,{}^*)$, which is the vector space $\mathbb R^4$ together with an involution ${}^*:\mathbb H\to \mathbb H$ and an algebra product $\cdot \star\cdot :\mathbb H\times \mathbb H\to \mathbb H$. Let $a=(a_0,a_1,a_2,a_3)$ and $ b=(b_0,b_1,b_2,b_3)$ be two elements in $\mathbb H$. Then $\star$ is defined by
\begin{align}
(a\star b)_0=& a_0b_0-a_jb_j\nonumber\\
(a\star b)_i=& a_0b_i+a_i b_0 + \epsilon_{ijk}a_jb_k,\nonumber
\end{align}
where the indices $i,j,k$ run over $\{1,2,3\}$ and $\epsilon_{ijk}$ is the Levi-Civita symbol. It is totally antisymmetric in its indices. The involution acts as $(a_0,a_1,a_2,a_3)\mapsto (a_0,-a_1,-a_2,-a_3)$. The 3-sphere $S^3$ can be embedded into $\mathbb H$ by $U\mathbb H=\{x\in\mathbb H : ||x||=1\}$. To each unit quaternion $x\in U\mathbb H$ one can associate an element in $SO(3)$, acting on purely imaginary quaternions $u\in\mathfrak I\mathbb H=\{a\in\mathbb H:a_0=0\}\cong \mathbb R^3$ by
\begin{equation*}
u\mapsto R_x(u)=  x\star u \star x^* \in \mathfrak I \mathbb H.
\end{equation*}
One can show that the mapping $x\mapsto R_x$ is a smooth, nondegenerate, two-to-one, surjective assignment of any $x$ to an element of $SO(3)$ and that $S^3$ is in fact the double cover of $SO(3)$. Hence there is a collection of submersions $\pi_u:S^3\to S^2$ indexed by vectors $u\in S^2$ that act like $\pi_u(x)=R_x(u)$. Choosing $u=(0,0,1)$ and setting $\pi=\pi_u$ we get
\begin{equation}
 \pi_i(x)=(x_0^2-x_jx_j)\delta_{i3}+2x_0\epsilon_{ij3}x_j + 2x_3x_i \label{eq:quotientmap}
\end{equation}
as one example of a Hopf map. Alternatively, one can describe this map as the quotient of a $U(1)$ action on $S^3\cong U\mathbb H$. We use the abreviation $\mathbb I=(1,0,0,0)$, $I=(0,1,0,0)$, $J=(0,0,1,0)$, and $K=(0,0,0,1)$. They satisfy the quaternion algebra $I\star I= J\star J=K\star K=I\star J\star K =-\mathbb I$. The $U(1)$ action 
\begin{equation}
(\text{e}^{Kt},x)\mapsto \text{e}^{Kt}\star (x_0+Kx_3)+ \text{e}^{-Kt}\star J\star (x_2-Kx_1)\label{eq:U1action} 
\end{equation}
 is generated by the vector field $w(x)=(-x_3,x_2,-x_1,x_0)$. We now show that $\pi$ is the quotient map of the $U(1)$-action \eqref{eq:U1action}. The differential of \eqref{eq:quotientmap} is given by
 \begin{equation*}
  (D\pi)_{i\mu}=2(x_0\delta_{\mu 0}-x_j \delta_{\mu j})\delta_{i 3}+2(x_j \delta_{\mu 0}+x_0\delta_{\mu j})\epsilon_{ij3}+2\delta_{\mu 3}x_i+2 \delta_{\mu i} x_3.
 \end{equation*}
A calculation reveals that $D\pi w=0$, and $w$ spans $\ker D\pi$ since $\pi$ is a submersion and $\ker D\pi$ is one-dimensional.

Having introduced the lumping map $\pi:S^3\to S^2$ in the framework of quaternions and the Lie algebra action, generated by $w$, we now proceed with the example. There is a  collection of vector fields $v_c(x)=c\star x$, indexed by $c\in\mathfrak I\mathbb H$, given by
\begin{equation*}
 (v_c)_\mu(x)=-\delta_{\mu 0} c_j x_j +\delta_{\mu j}c_j x_0 +\delta_{\mu j}\epsilon_{jkl}c_k x_l,
\end{equation*}
which is exactly lumpable for $\pi$ as in \eqref{eq:quotientmap}. We will now show that this follows from Proposition \ref{pr:main5}.
 The Lie group   $U(1)$ is compact; so, its action is proper and, since $w$ is nowhere vanishing, it is also free. We check whether $\mathcal L_{v_c}w\in\Gamma^{\infty}(X,\ker D\pi)$:
 \begin{align}
  [w,v_c]_\alpha=&w_\mu\frac{\partial (v_c)_\alpha}{\partial x_\mu}-(v_c)_\mu\frac{\partial w_\alpha}{\partial x_\mu}\nonumber\\
  =& +(x_1c_2-x_0c_3-x_2c_1+x_0c_3-\epsilon_{3kl}x_kc_l)\delta_{\alpha 0}\nonumber\\
  &-(x_0c_2+\epsilon_{2kl}x_l c_k)\delta_{\alpha 1}+(x_0c_1\epsilon_{1kl}x_lc_k)\delta_{\alpha 2}+(x_jc_j-x_0c_0)\delta_{\alpha 3}
  \nonumber\\&-(x_3c_j-\epsilon_{jk1}x_2c_k+\epsilon_{jk2}x_1c_k-\epsilon_{jk3}x_0c_k)\delta_{\alpha j}\nonumber\\
  =&\;\,0\;.\nonumber
 \end{align}
So we invoke Proposition \ref{pr:main5} which implies lumpability.
In fact,
\begin{equation*}
 (D\pi v_c)_i(x)=2\epsilon_{ijk}c_j\pi_k(x).
\end{equation*}
The lumped dynamics for the vector field that generates quaternion rotations $v_c=\frac{d}{dt}\big|_0\text{e}^{tc}\star x=c\star x$ under the quotient map $\pi$ is $\tilde v_c(y)=2\,c\times y$. Clearly it runs tangent to the sphere since $\tilde v_c\cdot y=0$ for $y\in S^2$.

\begin{myprop}
\label{pr:invariantsets}
Exact lumpings preserve invariant sets.
\end{myprop}
\begin{proof}
 Let $\mathcal A$ be a forward (resp., backward) invariant set, i.e. for all $t\geq 0$ the flow preserves the invariant set $\Phi_t\mathcal A\subseteq \mathcal A$ (resp., $\Phi_{-t}\mathcal A\subseteq \mathcal A$). 
After a projection with the lumping map, $\pi\circ \Phi_t \mathcal A\subseteq \pi \mathcal A$ (resp., $\pi\circ \Phi_{-t} \mathcal A\subseteq \pi \mathcal A$). Invoking the lumping condition from Proposition \ref{pr:main2} yields
 \begin{equation*}
  \tilde \Phi_t \circ \pi \mathcal A\subseteq \pi \mathcal A \qquad (\text{resp., }\,\tilde\Phi_{-t} \circ \pi \mathcal A\subseteq \pi \mathcal A);
 \end{equation*}
so, $\pi \mathcal A$ is a forward (resp., backward) invariant set of $\tilde \Phi_t$.
\end{proof}

This property can be exploited to determine invariant sets of the dynamics by finding the stationary points of a 1-dimensional exact lumping.
We conclude with a final example which also illustrates this feature. For a set of real coefficients $a_i$ which are not all zero, the logistic dynamics
\begin{equation*}
\dot x_i=x_i(1- a_j x_j), \qquad i=1,\dots,n,
\end{equation*}
has two invariant sets $\Omega_0=\{a_jx_j=1\}$ and $\Omega_1=\{x=0\}$ that are preserved under the lumping map $\pi(x)=a_jx_j$.
With $v_i=x_i(1- a_j x_j)$ we calculate
\begin{equation*}
D \pi v(x)=\frac{\partial \pi}{\partial x_i}v_i(x)=a_ix_i(1-a_jx_j)
\end{equation*}
and find that $\tilde v(y)=y(1-y)$ is the lumped dynamics. Hence by Proposition \ref{pr:invariantsets}, $\pi \Omega_0$ and $\pi \Omega_1$ are invariant under $\tilde v$.

\bigskip\bigskip

\noindent\textbf{Acknowledgement.} 
The research leading to these results has received funding from the European Union's Seventh Framework Programme (FP7/2007-2013) under grant agreement no. 318723 (MATHEMACS). L.H. acknowledges funding by the Max Planck Society through the IMPRS scholarship. 

\bigskip\bigskip


\bibliographystyle{unsrt}

\end{document}